\documentclass[11pt,reqno]{amsart}
\usepackage[utf8]{inputenc}
\usepackage[maxbibnames=50, backend=bibtex]{biblatex}
\usepackage{amsmath,amsthm,amssymb,physics}
\usepackage[vcentermath,enableskew]{youngtab}
\usepackage{tikz}
\usepackage{geometry}
\usepackage[colorlinks,breaklinks]{hyperref}
\usepackage{xurl}
\usepackage{tikz-cd}
\usepackage{comment}
\usepackage{mathrsfs}
\usepackage{float}
\usepackage{graphicx}
\usepackage{algpseudocode}
\usepackage{algorithm}

\title{4-strand Burau is unfaithful modulo $5$}

\author{Joel Gibson}
\address{J.~Gibson: Sydney, Australia}
\email{joel@jgibson.id.au}
\date{}

\author{Geordie Williamson}
\address{G.~Williamson: School of Mathematics and Statistics, University of Sydney, Australia}
\email{g.williamson@sydney.edu.au}
\date{}

\author{Oded Yacobi}
\address{O.~Yacobi: School of Mathematics and Statistics, University of Sydney, Australia}
\email{oded.yacobi@sydney.edu.au}
\date{}

\def\GL{\mathrm{GL}}

\newcommand{\sB}{\mathsf{B}}

\newcommand{\sR}{\mathscr{R}}

\newcommand{\FF}{\mathbb{F}}

\newcommand{\NN}{\mathbb{N}}

\newcommand{\ZZ}{\mathbb{Z}}

\newcommand{\pl}{\mathsf{projlen}}

\renewcommand{\phi}{\varphi}

\renewcommand{\tilde}[1]{\widetilde{#1}}

\makeatletter
\def\Ddots{\mathinner{\mkern1mu\raise\p@
\vbox{\kern7\p@\hbox{.}}\mkern2mu
\raise4\p@\hbox{.}\mkern2mu\raise7\p@\hbox{.}\mkern1mu}}
\makeatother

\newcommand{\val}{\text{val}}
\newcommand\into\hookrightarrow
\newcommand\onto\twoheadrightarrow

\newcommand{\nc}{\newcommand}
\nc{\la}{\lambda}
\nc{\Iso}{\mathsf{Iso}}
\nc{\Irr}{\mathsf{Irr}}
\nc{\Id}{\mathrm{Id}}

\numberwithin{equation}{section}
\newtheorem{Theorem}[equation]{Theorem}
\newtheorem{Proposition}[equation]{Proposition}

\newtheorem{Remark}[equation]{Remark}

\def\gchange#1{{\color{blue}#1}}

\begin{document}

\begin{abstract}
We introduce a new algorithm for finding kernel elements in the
   Burau representation. Our algorithm applies
  reservoir sampling to a statistic on matrices which
  is closely correlated with Garside length. Using this we exhibit an explicit kernel element in
  the Burau representation on 4-strands reduced modulo 5. 
\end{abstract}

\maketitle

\section{Introduction}

The Burau representation of the $n$-strand braid group $\pi_n:B_n \to \GL_{n-1}(\ZZ[v,v^{-1}])$, plays a prominent role in many applications of braid groups, including to geometry, topology and mathematical physics \cite{Birman}.
A key question underlying many of these developments is determining
when  $\pi_n$ is faithful.

It is not difficult to prove that $\pi_3$ is faithful (cf. Propostion \ref{prop:3burau}), but despite
intensive study, there wasn't substantial progress on this question until 1991, when
Moody proved that $\pi_n$ is not faithful for $n\geq9$ \cite{Moody}.
His work was developed further by Long and Paton to show
unfaithfulness for $n \geq 6$ \cite{LP93}. Bigelow extended this also to $n=5$ \cite{Bigelow99}.

The case $n=4$ remains open, and has attracted considerable attention. One reason for this is that a negative answer   
  provides a non-trivial knot with trivial Jones polynomial, answering a central question in knot theory
  \cite{Bigelow02, Ito}. 

Henceforth we set $\pi=\pi_4$.
Cooper and Long approached the faithfulness of $\pi$ by reducing
modulo primes.  That is,  they considered the composition of $\pi$
with the natural map $\GL_{3}(\ZZ[v,v^{-1}]) \to
\GL_{3}(\FF_p[v,v^{-1}])$, where $\FF_p$ is the field with $p$
elements \cite{CL97}.  Following their work, we denote this
representation $\pi\otimes \FF_p$. The study of the faithfulness of
$\pi \otimes \FF_p$ is interesting for several reasons:
\begin{enumerate}
\item Examples when $\pi \otimes \FF_p$ is unfaithful can give
    insights into the difficulty of establishing the faithfulness of
    $\pi$.
\item It might be the case that $\pi \otimes \FF_p$ is unfaithful modulo all
    primes, yet $\pi$ is faithful. Controlling the kernel modulo infinitely many primes might
    provide a route to establishing
    faithfulness of $\pi$.
\end{enumerate}

One of the main results in \cite{CL97} is that $\pi\otimes\FF_2$ is not faithful.  In \cite{CL98} the same authors extended their methods to show that $\pi\otimes\FF_3$ is also not faithful.  They noted their program runs into obstacles at $p=5$: 
``This case remains open and has some features which suggest it may be different to the first two primes." Our main result resolves this case.

\begin{Theorem}
The representation $\pi\otimes \FF_5$ is not faithful.
\end{Theorem}

Some remarks on our main theorem:
\begin{enumerate}
\item Our approach, which we discuss in the next section, is heavily
  computational and is quite different from the existing
  methods. A complete implementation of our algorithm (in
    Python) is available on github \cite{github}.
\item We discover several kernel elements modulo $5$, the smallest of
  which is of Garside length 54\footnote{see \S \ref{sec:recollections} for the definition of Garside length}. We note that a straightforward computation
  shows that there are $10^{40}$ elements of this Garside length, thus
  finding this element by brute force search is not feasible.
\item We have made extensive searches with our algorithm to
    try to discover a kernel element of $\pi \otimes \FF_7$, without
    success. We are also unable to rediscover Bigelow's (integral)
    kernel elements for $n=5$ and $n=6$, which indicates the
    limitations of our algorithm.
\item Recently, Datta proved that $\pi$ is ``generically faithful'', in
  the sense that almost all braids are not in the kernel of $\pi$ as
  their length tends to infinity \cite{Datta}.  Datta also proves that
  this property extends to $\pi\otimes \FF_p$ for $p>3$.  Combined
  with our work, $\pi\otimes \FF_5$ therefore provides an example of a
  generically faithful representation which is not faithful.
\end{enumerate}

Our original goal was to use Machine Learning techniques to attack
this problem. Thus far we have not been able to get anything along
these lines to work---we outline our attempts in \S \ref{sec:future-directions}. We
do believe though that this is an excellent hard test case for Machine Learning
methods in pure mathematics. Our implementation \cite{github} is
intended to make it as straightforward as possible for interested
researchers to try Machine Learning techniques on this problem.

\subsection{Acknowledgements} The
observation that Garside length is closely correlated with $\pl$ (cf. \eqref{eq:projlen}) was
made in joint work of the third author with Nicolas Libedinsky. (It
has likely also been noticed by other researchers.) We would also like
to thank Robert Bryant, Giles Gardem, Libedinsky and Nolan Wallach for useful discussions. The second two
authors are supported by the Australian Research Council under the
grant DP230100654.

\section{Recollections}\label{sec:recollections}
We recall some basic facts about the braid group.
A presentation of $B_n$ in terms of the Artin generators is given by 
$$
B_n=\Big\langle \sigma_1,\hdots,\sigma_{n-1} \;|\; \sigma_i\sigma_j=\sigma_j\sigma_i \text{ if } |i-j|>1 \text{ and } \sigma_i\sigma_j\sigma_i=\sigma_j\sigma_i\sigma_j \text{ if }  |i-j|=1 \Big\rangle.
$$
On these generators, the \textbf{(reduced) Burau representation} $\pi_n:B_n \to \GL_{n-1}(\ZZ[v,v^{-1}])$ is defined as follows: 
\begin{align*}
\sigma_1 \mapsto \begin{pmatrix}
-v^2 & -v \\
0 & 1 
\end{pmatrix} \oplus& I_{n-3}, \;\;\; \sigma_{n-1} \mapsto I_{n-3}\oplus \begin{pmatrix}
1 & 0 \\
-v & -v^2
\end{pmatrix} \\
\sigma_i \mapsto I_{i-1}\oplus& \begin{pmatrix}
1 & 0 & 0\\
-v & -v^2 & -v \\
0 & 0 & 1
\end{pmatrix} \oplus I_{n-i-3}
\end{align*}
for $1<i<n-1$.
We call $\pi_n(\sigma)$ the \textbf{Burau matrix} of $\sigma$, and its reduction modulo a prime $p$, its $p$-Burau matrix.

Let $S_n$ denote the symmetric group with simple generators $s_1,\hdots,s_{n-1}$.  
For $w \in S_n$ write $\tilde{w} \in B_n$ for the corresponding positive lift in the braid group.  Explicitly, if $w=s_{i_1}s_{i_2}\cdots$ is a reduced expression, then $\tilde{w}=\sigma_{i_1}\sigma_{i_2}\cdots$.  For the longest element we write $\Delta=\tilde{w_0}$. Note that
\begin{align}\label{eq:delta-matrix}
    \pi_n(\Delta)=(-1)^{n+1}\begin{pmatrix}
0 &  & v^n\\
 & \Ddots &  \\
v^n &  & 0
\end{pmatrix}
\end{align}
Let $\mathscr{R}(w)$ (respectively $\mathscr{L}(w)$) denote the right (respectively left) descent set of $w\in S_n$.

A braid $\sigma\in B_n$ has a \textbf{Garside normal form (GNF)}
\begin{align}\label{eq:gnf}
    \sigma=\Delta^d \tilde{w_{1}}\cdots \tilde{w_{\ell}},
\end{align}
where $d\in \ZZ$, $w_{1} \neq w_0$, $w_{\ell}\neq e$, and for every $k$, $\mathscr{R}(w_{k}) \supseteq \mathscr{L}(w_{k+1})$.  This expression is unique, and $\ell_G(\sigma):=\ell$ is the \textbf{Garside length} of $\sigma$.
For $k\leq \ell$ we'll be interested in the associated \textbf{Garside prefix} of $\sigma$: $\Delta^d \tilde{w_{1}}\cdots \tilde{w_{k}}$.
A simple but useful observation: 
if $u\in S_n$ satisfies $\mathscr{R}(w_\ell) \supseteq \mathscr{L}(u)$, then the GNF of
$\sigma\tilde{u}$ is obtained  by concatenation:
\begin{align*}
    \sigma\tilde{u}=\Delta^{d} \tilde{w_{1}}\cdots \tilde{w_{\ell}}\tilde{u}.
\end{align*}
We term $\tilde{u}$ a
\textbf{Garside suffix} of $\sigma$. 

\section{Basic idea} Suppose one can find a statistic on  matrices
which detects  Garside length.  In other words, we have a function
$\mathsf{p}:\GL_{n-1}(\ZZ[v,v^{-1}]) \to \NN$ such that
$\mathsf{p}(\pi_n(\sigma))=C\cdot \ell_G(\sigma)$ for some constant $C
\gchange{>} 0$. Then this implies the faithfulness of $\pi_n$: 
$$
\pi_n(\sigma)=\Id \;\Longrightarrow\; \mathsf{p}(\pi_n(\sigma))=0 \;\Longrightarrow\; \ell_G(\sigma)=0 \;\Longrightarrow\; \sigma=\Delta^d
$$ 
for some $d \in \ZZ$, and $\pi_n(\Delta^d)=\Id$ if and only if
$d=0$.

On the other hand, if $\mathsf{p}$ is correlated with $\ell_G$ only for generic braids, then one can try to use  $\mathsf{p}$ as an ansatz to find kernel elements: braids with surprisingly low $\mathsf{p}$-value might point towards elements in the kernel. We explain both of these situations, beginning with the simple case of the $3$-strand Burau representation.  

First we define the statistic which we'll be using
throughout. 
Given $A$, a non-zero $r \times s$ array with entries in 
  Laurent polynomials in $v$, let $\deg(A)$ be the highest power of $v$ that occurs in an entry of $A$, and let $\val(A)$ be the lowest power.
We set
\begin{align}\label{eq:projlen}
\pl(A)=\deg(A)-\val(A).
\end{align} 
Note that multiplication of a Burau matrix by $\Delta$ does not affect $\pl$: 
\begin{align}\label{eq:projlen-and-delta}
\pl(\pi_n(\sigma))=\pl(\pi_n(\sigma)\Delta)=\pl(\Delta\pi_n(\sigma)).
\end{align}
\begin{Proposition}\label{prop:3burau}
    For any $\sigma\in B_3$ we have $\pl(\pi_3(\sigma))=2\ell_G(\sigma)$.  Hence, the $3$-strand Burau representation is faithful.
\end{Proposition}    

\begin{proof}
    Let $\sigma \in B_3$ have GNF given by \eqref{eq:gnf} and let $\pi_3(\sigma)=[c_1,c_2]$, where $c_i$ are the column vectors in the matrix.  Then  the following trichotomy holds:
    \begin{enumerate}
        \item $\deg(c_1)=\deg(c_2), \val(c_1)=\val(c_2)$, and $\sigma=\Delta^d$ for some $d\in \ZZ$.  
        \item $\deg(c_1)>\deg(c_2), \val(c_1)>\val(c_2)$, and $\sR(\sigma)=\{s_1\}$.
        \item $\deg(c_1)<\deg(c_2), \val(c_1)<\val(c_2)$, and $\sR(\sigma)=\{s_2\}$.
    \end{enumerate}
   Moreover, in each case $\pl(\pi_3(\sigma))=2\ell_G(\sigma)$.

To see this, we induct on $\ell=\ell_G(\sigma)$.  If $\ell =0$ then we are in the first case of the trichotomy.  Note that in this case  $0=\pl(\pi_3(\sigma))=2\ell_G(\sigma)$.   Now let $\ell\geq1$, and suppose that $\sigma':=\Delta^d \tilde{w_{1}}\cdots \tilde{w_{\ell-1}}$ falls into the second case of the trichotomy.  Setting $\pi_3(\sigma')=[c_1',c_2']$, we then have that $\deg(c'_1)>\deg(c'_2), \val(c'_1)>\val(c'_2)$, $\sR(\sigma')=\{s_1\}$, and $\pl(\pi_3(\sigma'))=2\ell_G(\sigma')$.  The condition $\sR(\tilde{w_{\ell-1}}) \supseteq \sR(\tilde{w_{\ell}})$ forces $w_\ell=s_1$ or $w_\ell=s_1s_2$. If $w_\ell=s_1$ then 
$
[c_1,c_2]=[-v^2c_1',-vc'_1+c'_2],
$
and we are in case (2).  Otherwise, $w_\ell=s_1s_2$ and 
$
[c_1,c_2]=[-vc_2',v^3c'_1-v^2c'_2],
$
so we are in case (3). Either way we  have the equality $\pl(\pi_3(\sigma))=2\ell_G(\sigma)$.  The other cases follow similarly.
\end{proof}

Moving onto the $4$-strand braid group, it is easy to see that a
straightforward analogue of Proposition \ref{prop:3burau} does not
hold.  For instance, taking $\sigma=\tilde{u}$, where $u=s_1s_2s_1s_3$
we have that 
\begin{align*}
    \pi_4(\sigma)=\begin{pmatrix}
0 & 0 & -v^4\\
v^3 & v^2 & v^3 \\
0 & -v & -v^2
\end{pmatrix}
\end{align*}
We see that $\pl(\pi_4(\sigma))=3$ is not even, and  one cannot read off $\sR(u)=\{s_1,s_3 \}$ directly from the columns achieving  $\deg(\pi_4(\sigma))$. Nevertheless, we can ask whether generically $\pl$ is a good heuristic for Garside length. 

The following table shows the result of sampling $1000$ random braids in $B_4$ of each  Garside length, up to $50$.  Each braid is depicted by a blue dot in the plane, where the $x$-axis is labelled by the Garside length of the braid, and the $y$-axis is labelled by $\frac{1}{2}\pl$.  The $y=x$ line is drawn in red. 
\begin{align*}
\includegraphics[scale=.51]{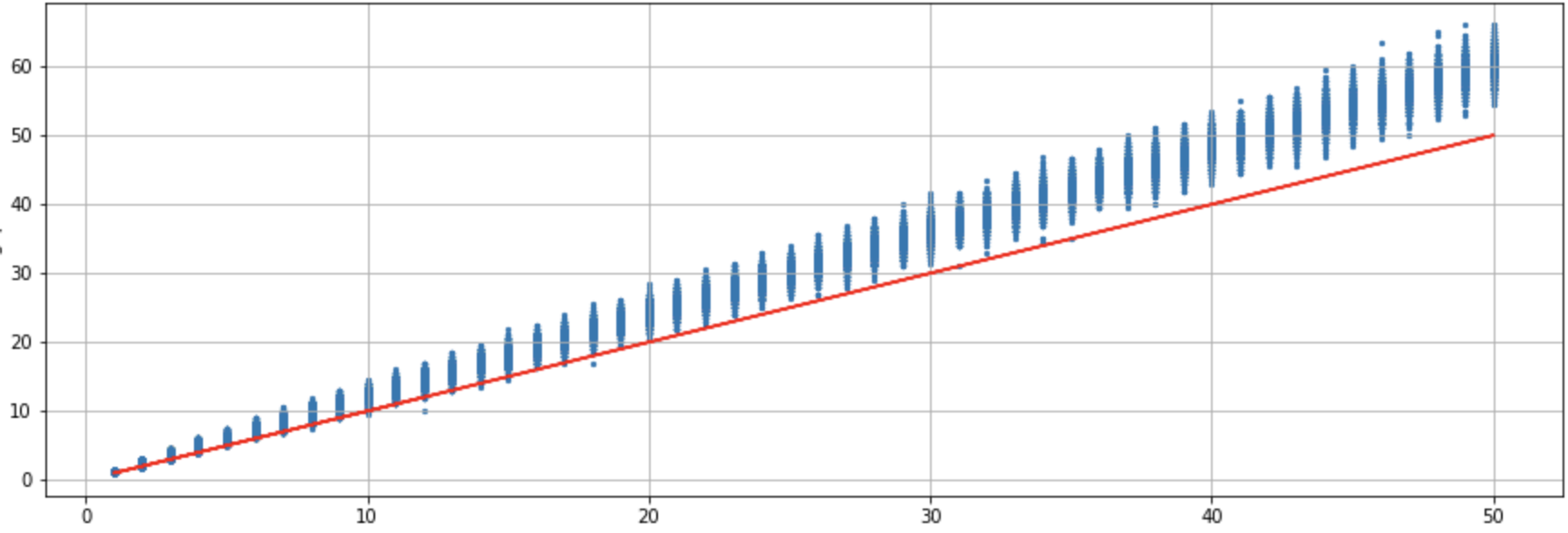}
\end{align*}

We see that it is almost always the case that $\frac{1}{2}\pl$ is larger than the Garside length. Moreover, there appears to be a   strong linear correlation
between these two statistics for generic braids. However, there
are a few exceptions corresponding to blue dots below the red
line.

Now consider an element
  \[ \sigma=\Delta^d \tilde{w_{1}}\cdots \tilde{w_{\ell}}\]
    in the kernel of some Burau representation. Because $\pl$ of a
    kernel element is zero and $\pl$ is unaffected by multiplication
    by $\Delta$  \eqref{eq:projlen-and-delta}  we have
    \[
0 = \pl(\sigma) = \pl( \tilde{w_{1}}\cdots \tilde{w_{\ell}})
      \]
      It is now reasonable to suspect that the  projlens of Garside prefixes of $\tilde{w_{1}}\cdots \tilde{w_{\ell}}$:
\begin{align}\label{eq:trajectory}
\pl(\tilde{w_{1}}), \pl(\tilde{w_{1}}\tilde{w_{2}}), \dots, \pl(
\tilde{w_{1}}\cdots \tilde{w_{\ell-1}}), \pl( \tilde{w_{1}}\cdots
\tilde{w_{\ell}})=0.
\end{align}
are typically smaller than those of a generic braid of the same
length. Thus searching for GNFs with low projlen might point us in the
direction of kernel elements.

\section{The algorithm}
We'll now describe our algorithm for sampling braids with
low $\pl$.  We picture the set-up for our search as placing a
``bucket'' at each point in $\ZZ_{\geq0}\times\ZZ_{\geq0}$.
Initially, these buckets are empty and over time they will be filled
with Garside normal forms of $4$-strand braids, and their Burau matrices.  The bucket $\sB_{(\ell,m)}$ at the point $(\ell,m)$ will only contain braids $\sigma$ such that $\ell_G(\sigma)=\ell$ and $\pl(\pi(\sigma))=m$.  

The method by which we fill the buckets is called \textbf{reservoir
  sampling} \cite{Vitter}.  This is a well-known algorithm which
selects a random sample of $k$ elements without repetition from a set
of $N$ elements which arrive sequentially, and where the
value of $N$ is unknown in advance. Typically, $N$ is enormous. Thus it
  is not possible to store all elements seen and the sampling is done in one pass.  At any
point in time, the selection of $k$ elements, i.e. the ``reservoir'',
should be uniformly distributed among all elements seen thus
far.

We recall the  idea in more detail since it is so fundamental to our approach.  
Suppose we're sampling elements from a set $X$ and that the
  elements arrive as a sequence $x_1, x_2, \dots, x_N$.

  Let us first consider the case $k=1$.  Our reservoir then
  consists of a single element, which at the $i$-th step we will
  denote $r_i$.  At the first step we let $r_1=x_1$.  In the
  second step, we replace $r_1$ with $x_2$ with probability
  $\frac{1}{2}$, otherwise $r_2 = r_1$. In the third step we replace
  $r_2$ with $x_3$ with probability $\frac{1}{3}$, otherwise $r_3 =
  r_2$. Note that indeed, for every $i$,  $r_i$ is uniformly
  distributed among $x_1,\hdots,x_i$.

If $k>1$, then the reservoir is a $k$-subset of $X$, which at the
$i$-th step we will denote $R_i=\{r_1^i,\hdots,r_k^i\}$.  First we just fill the reservoir: $r_j^1=x_j$ for $j\leq k$.  At the next step we randomly  generate a number $j\in\{1,\hdots,k+1\}$, and if $j\leq k$ replace the $j$-th element of $R_{1}$ by $x_{k+1}$:
\begin{align*}
    r_q^{2}:=\begin{cases}
        x_{k+1} \text{ if } q=j, \\
        r_q^{1} \text{ otherwise. }
    \end{cases}
\end{align*}
Otherwise,  $R_{2}=R_1$. Next, randomly generate a number $j\in\{1,\hdots,k+2\}$, and if $j\leq k$
replace the $j$-th element of $R_{2}$ by $x_{k+2}$, and so on.  At the $i$-th step, we have that $R_i$ is a
uniformly distributed $k$-subset of $\{x_1,\hdots,x_{i+k-1}\}$.

Going back to our set-up, we will now describe an adaptation of the above ideas to place braids in appropriate buckets.  At the $\ell$-th step in our algorithm, we will  be placing braids in buckets at points $(\ell,m)$ for various $m$. We  fix ahead of time an upper bound $k$, which is the maximum number of braids contained in a single bucket.  Set $\sB_\ell=\bigcup_{m\geq0}\sB_{(\ell,m)}$.  

To begin place the identity $e\in B_4$ in $\sB_{(0,0)}$.  Now let $\ell>0$, and at the $\ell$-th step proceed as follows:

\begin{algorithm}[hbt!]
\begin{algorithmic}
\State Set $N_{m}=0$ for every $m$ \Comment{Counts how many braids we've tried to place in $\sB_{(\ell,m)}$}
\For{$\sigma \in \sB_{\ell-1}$}
\For{$\tilde{u}$ a Garside suffix of $\sigma$}
\State $m=\pl(\sigma\tilde{u})$
\State $M=|\sB_{(\ell,m)}|$
\If{$M<k$}
    \State place $\sigma\tilde{u}$ and its Burau matrix in $\sB_{(\ell,m)}$
\Else{ randomly generate a number $j \in \{1,\hdots,N_{m} \}$}
    \If{$j<k$}
        \State replace $j$-th element of $\sB_{(\ell,m)}$ by $\sigma\tilde{u}$ and its Burau matrix
    \Else{ discard $\sigma\tilde{u}$}
    \EndIf
\EndIf
\State Set $N_{m}:=N_{m}+1$
\EndFor
\EndFor
\end{algorithmic}
\end{algorithm}
Note that since we are storing the Burau matrices along the way, it is easy for us to compute $\pl(\sigma\tilde{u})$.  Also, we can easily modify this algorithm to study the Burau representation modulo primes (we simply remember the $p$-Burau matrices instead).

\section{Main result}
We'll now describe the most interesting output of our algorithm:
 in the case $p=5$ we obtained several elements in the kernel of $\pi\otimes\FF_5$. We stress that these are the first known kernel elements in Burau representations $\pi\otimes\FF_p$ with $p>3$, and this represents the first explicit progress on this question in 25 years! 

 The smallest we found, $\kappa$, has Garside length $54$. We have
 \[
\kappa = \Delta^{-27}*\sigma
   \]
where $\Delta$ denotes the Garside element and $\sigma$ is the following word in the Artin generators:
\begin{gather*} 
  1, 3, 1, 3, 1, 3, 2, 2, 1, 3, 3, 2, 2, 1, 3, 3, 2, 2, 2, 1, 3, 1,
  3, 2, 1, 2, 1, 3, 1, 3, 2, 1, 2, 1, 3, 1, 3, 2, 2, 2, 1, 3, 3, \\
  2, 2, 1, 3, 3, 2, 2, 1, 3, 1, 3, 1, 2, 3, 2, 1, 1, 3, 2, 1, 2, 1, 3,
  1, 3, 2, 1, 2, 1, 3, 1, 3, 2, 2, 1, 3, 2, 2, 1, 3, 2, 2, 2, 1, \\
  3, 3, 2, 2, 1, 3, 3, 2, 2, 1, 3, 1,  2, 3, 2, 1, 1, 3, 2, 1, 2, 1,
  3, 1, 3, 2, 1, 2, 1, 3, 1, 2, 3, 2, 1, 1, 3, 2, 2, 1, 3, 3, 2, \\
  2, 1, 3, 3, 2, 2, 2, 1, 3, 2, 2, 1, 3, 1, 2, 3, 2, 2, 1, 3, 1, 2, 3,
  2, 2, 1, 3, 1, 2, 3, 2, 1
\end{gather*}
The length of $\sigma$ in the Artin generators is 162.  

The following plot depicts the journey of $\sigma$ through the various buckets, as described in \eqref{eq:trajectory}: 
\begin{align*}
\includegraphics[scale=.35]{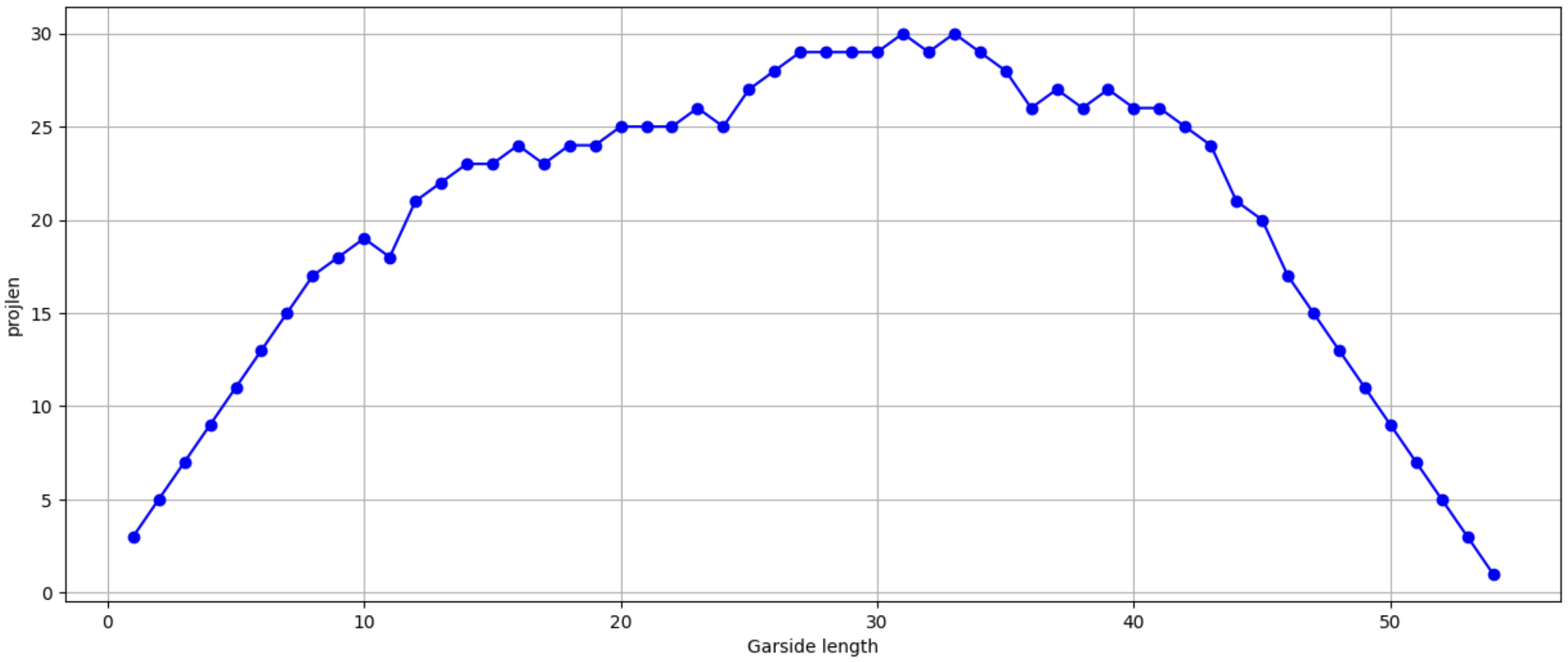}
\end{align*}
The $x$-axis is labelled by Garside length, the $y$-axis by $\pl$\footnote{It's actually labelled by $\pl+1$ but we ignore this technicality.}, and each blue dot represents a bucket which contains a Garside prefix of $\kappa$.  

The rightmost dot corresponds to the bucket containing $\sigma$.  It's
interesting to note that this trajectory follows a generic path until
about Garside length $30$, where suddenly the reservoir sampling
starts finding elements with smaller than expected $\pl$.  Around
length $40$, the algorithm hits a ``point of no return'', and makes a
beeline for $\sigma$.

In case the reader is interested, here is the Garside normal form of $\kappa$:
\begin{gather*} \kappa = 
(\Delta^{-27};s_{0} s_{2}, s_{0} s_{2}, s_{0} s_{2} s_{1}, s_{1} s_{0}
s_{2}, s_{2} s_{1}, s_{1} s_{0} s_{2}, s_{2} s_{1}, s_{1}, s_{1} s_{0}
s_{2},s_{0} s_{2} s_{1} s_{0}, s_{1} s_{0} s_{2}, \\
s_{0} s_{2} s_{1} s_{0}, s_{1} s_{0} s_{2}, s_{0} s_{2} s_{1}, s_{1}, s_{1} s_{0} s_{2}, s_{2} s_{1}, s_{1} s_{0} s_{2}, s_{2} s_{1}, s_{1} s_{0} s_{2}, s_{0} s_{2}, s_{0} s_{1} s_{2} s_{1} s_{0}, s_{0} s_{2} s_{1} s_{0}, \\s_{1} s_{0} s_{2}, s_{0} s_{2} s_{1} s_{0}, s_{1} s_{0} s_{2}, s_{0} s_{2} s_{1}, s_{1} s_{0} s_{2} s_{1}, s_{1} s_{0} s_{2} s_{1}, s_{1}, s_{1} s_{0} s_{2}, s_{2} s_{1}, s_{1} s_{0} s_{2}, s_{2} s_{1}, s_{1} s_{0} s_{2}, \\s_{0} s_{1} s_{2} s_{1} s_{0}, s_{0} s_{2} s_{1} s_{0}, s_{1} s_{0} s_{2}, s_{0} s_{2} s_{1} s_{0}, s_{1} s_{0} s_{2}, s_{0} s_{1} s_{2} s_{1} s_{0}, s_{0} s_{2} s_{1}, s_{1} s_{0} s_{2}, s_{2} s_{1}, \\s_{1} s_{0} s_{2}, s_{2} s_{1}, s_{1}, s_{1} s_{0} s_{2} s_{1}, s_{1} s_{0} s_{2}, s_{0} s_{1} s_{2} s_{1}, s_{1} s_{0} s_{2}, s_{0} s_{1} s_{2} s_{1},s_{1} s_{0} s_{2}, s_{0} s_{1} s_{2} s_{1} s_{0})
\end{gather*}

Since randomness is built into the algorithm, every time we run it we
get different outcomes.  About $20\%$ of the time we actually find a
kernel element, and this takes about a couple of hours on a standard
personal computer.

Other runs of our algorithm discovered two other elements in the
kernel of Garside lengths 59 and 65 respectively. They are
\[
\kappa_1 = \Delta^{-29} \sigma_1 \quad \text{and} \quad \kappa_2 =
\Delta^{-33} \sigma_2
\]
where $\sigma_1$ and $\sigma_2$ are given by the following words in
the Artin generators, of lengths 174 and 198 respectively:
\begin{gather*}
    \sigma_1 ={\scriptstyle
(1, 2, 1, 3, 2, 2, 1, 3, 1, 3, 2, 2, 2, 1, 3, 1, 2, 2, 1, 3, 1, 2, 2,
1, 3, 1, 3, 1, 2, 3, 2, 1, 1, 2, 3, 2, 2, 1, 3, 1, 2, 3, 2, 2, 1, 3,
1, 3, 2, 2, 2, 1, 3, 1, 2, 2, 1, 3, 1, }\\ {\scriptstyle 2, 2, 2, 1,
3, 1, 2, 3, 2, 2, 1, 3, 1, 2, 3, 2, 2, 1, 3, 1, 2, 3, 2, 1, 1, 3, 1,
3, 1, 3, 2, 2, 1, 3, 1, 2, 2, 1, 3, 1, 2, 2, 2, 1, 3, 2, 2, 1, 3, 1,
3, 2, 1, 2, 1, 3, 1, 3, 2, 1, 2, 1, }\\{\scriptstyle 3, 1, 2, 3, 2, 1,
1, 3, 2, 2, 1, 3, 3, 2, 2, 1, 3, 3, 2, 2, 1, 3, 1, 2, 3, 2, 1, 1, 3,
2, 1, 2, 1, 3, 1, 3, 2, 1, 2, 1, 3, 1, 3, 2, 2, 1, 3, 2, 2, 2, 1, 3,
1)} \\
    \sigma_2 ={\scriptstyle
(1, 3, 1, 3, 1, 3, 2, 2, 1, 3, 3, 2, 2, 1, 3, 3, 2, 2, 1, 3, 1, 2, 3,
2, 1, 1, 3, 2, 1, 2, 1, 3, 1, 3, 2, 1, 2, 1, 3, 1, 3, 2, 2, 2, 1, 3,
3, 2, 2, 1, 3, 3, 2, 2, 1, 3, 1, 2, 3, 2, }\\ {\scriptstyle 1, 1, 3,
2, 1, 2, 1, 3, 1, 3, 2, 1, 2, 1, 3, 1, 3, 2, 2, 2, 1, 3, 3, 2, 2, 1,
3, 3, 2, 2, 1, 3, 1, 2, 3, 2, 1, 1, 3, 2, 1, 2, 1, 3, 1, 3, 2, 1, 2,
1, 3, 1, 3, 2, 2, 2, 1, 3, 3, 2, 2, 1, 3, 3, 2, 2, 1, 3, 1, }\\ {\scriptstyle 2, 3, 2,
1, 1, 3, 2, 1, 2, 1, 3, 1, 3, 2, 1, 2, 1, 3, 1, 3, 2, 2, 2, 1, 3, 3,
2, 2, 1, 3, 3, 2, 2, 1, 3, 1, 2, 3, 2, 1, 1, 3, 2, 1, 2, 1, 3, 1, 3,
2, 1, 2, 1, 3, 1, 3, 2, 2, 1, 3, 2, 2, 1, 3, 2, 2, 1, 3, 2) }
  \end{gather*}
  Here are the trajectories of $\sigma_1$ and $\sigma_2$ through the buckets:
  \begin{align*}
\includegraphics[scale=.2]{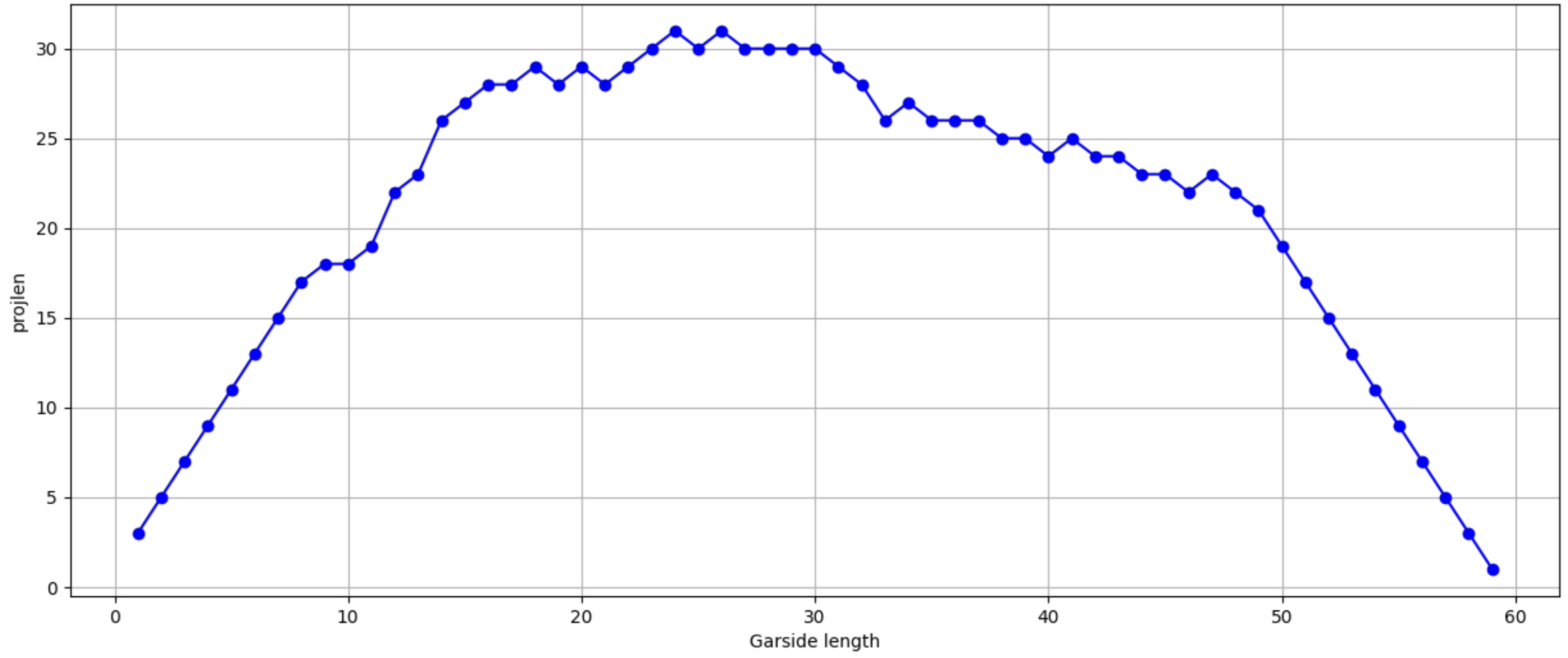}
\includegraphics[scale=.2]{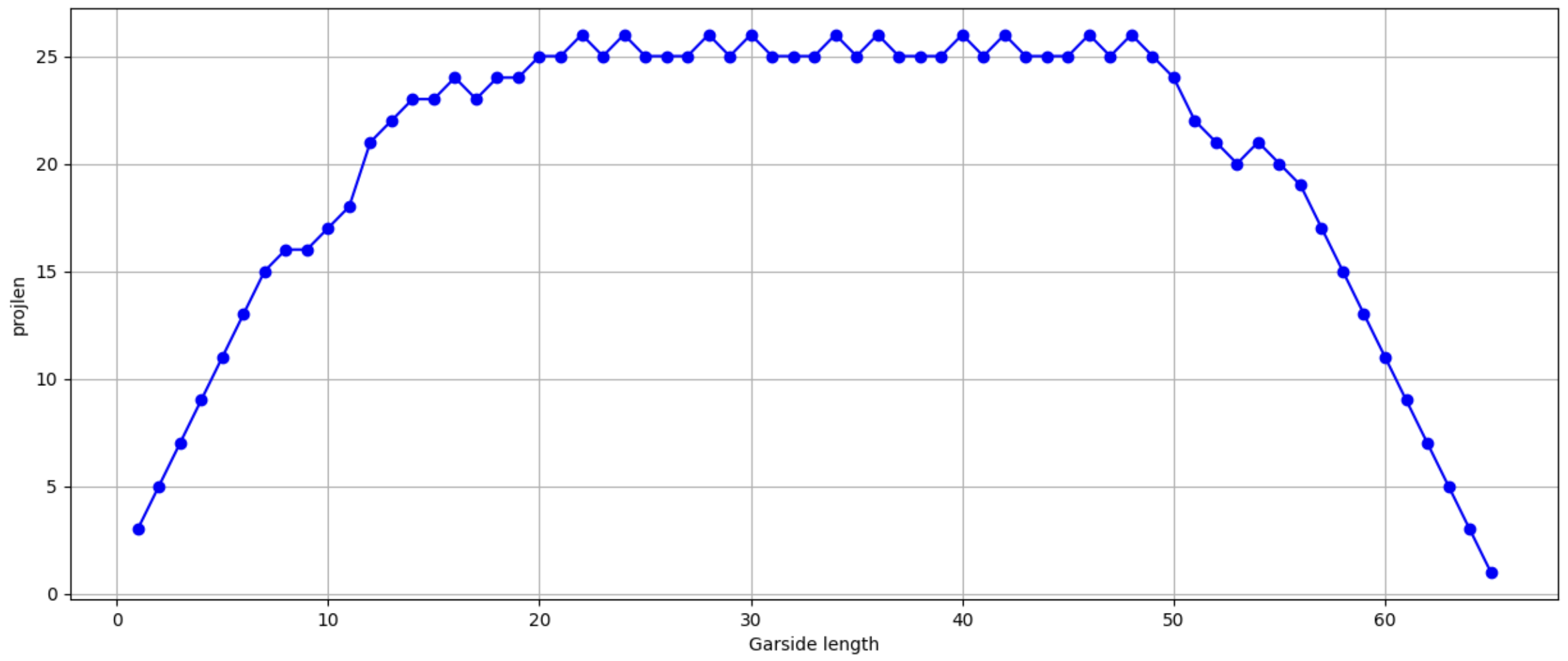}
  \end{align*}

  \begin{Remark}
    The reader wishing to explore more can experiment with the
    notebook ``p=5 kernel elements'' available at
    \cite{github}. We have also implemented a check of our results
    in Magma in the file ``Magma check.m'' which is also
    available at \cite{github}. It should be easy to modify the
    Magma file to suit any computer algebra system.
  \end{Remark}

\section{How good is our algorithm?}

We regard the fact that our algorithm is able to discover a kernel
element nearly 25 years after the last discovery of a kernel element
(for $p=3$ in \cite{CL98}) as some evidence that our
approach is interesting. On the other hand, it is interesting to ask:
can our algorithm rediscover known kernel elements?

\subsection{n=4} Here we see a plot Garside of length ($x$-axis)
against the minimal projlens found ($y$-axis),
over several runs of our algorithms modulo various integers $m =
2,3,4,5,6,7$:
\begin{align*}
\includegraphics[scale=.51]{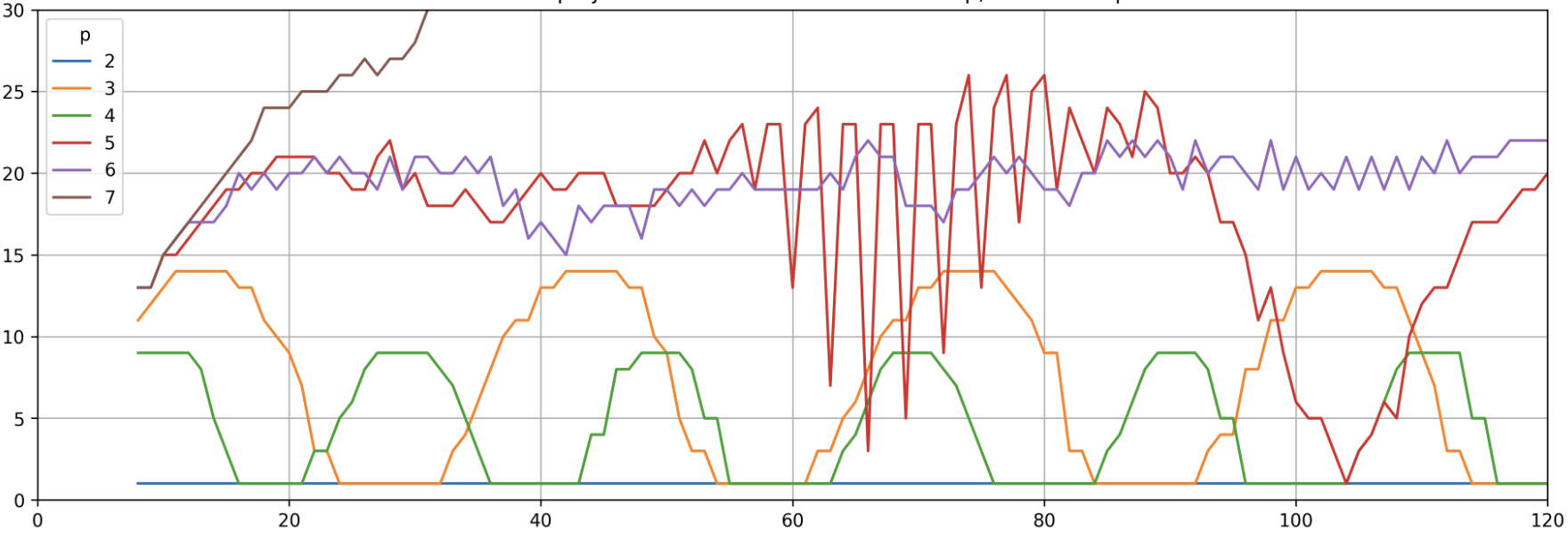}
\end{align*}
It immediately finds many kernel elements modulo $2$ and $3$. It also
finds kernel elements modulo $4$. Note also the jagged red line:
during these runs the algorithm found an element of Garside length
around 65 with low projlen, but did not find a kernel element of
this length. It did, however find a kernel element of Garside length
around 105 (where the red line finally touches the $x$-axis). Finally,
note that although we know that kernel elements exist modulo $6$, our
algorithm didn't find them. (The purple line.)

\subsection{n=5,6} Despite several attempts we have not been able to
recover Bigelow's (integral) kernel elements for $n=5$ and $n=6$. How
close are we to finding these elements?

It turns out that there is a beautiful and simple idea that allows us to accurately
answer this question. For concreteness, let us take $n=6$ in which case Bigelow's kernel
element is of Garside length $16$:
\[ \beta=\Delta^d \tilde{w_{1}} \tilde{w_2} \cdots \tilde{w_{16}}. \]
Now we can modify our code slightly to force it to add each of the
Garside divisors
\[
  \beta_1 = \tilde{w_{1}}, \quad 
  \beta_2 = \tilde{w_{1}} \tilde{w_2}, \quad \dots,  \quad
  \beta_{16} = \tilde{w_{1}} \tilde{w_2} \cdots \tilde{w_{16}}
\]
to their appropriate buckets during the reservoir search. In other
words, we peform the reservoir sampling as usual, however we force it
to successfully find $\beta$:
\begin{align*}
\includegraphics[scale=.51]{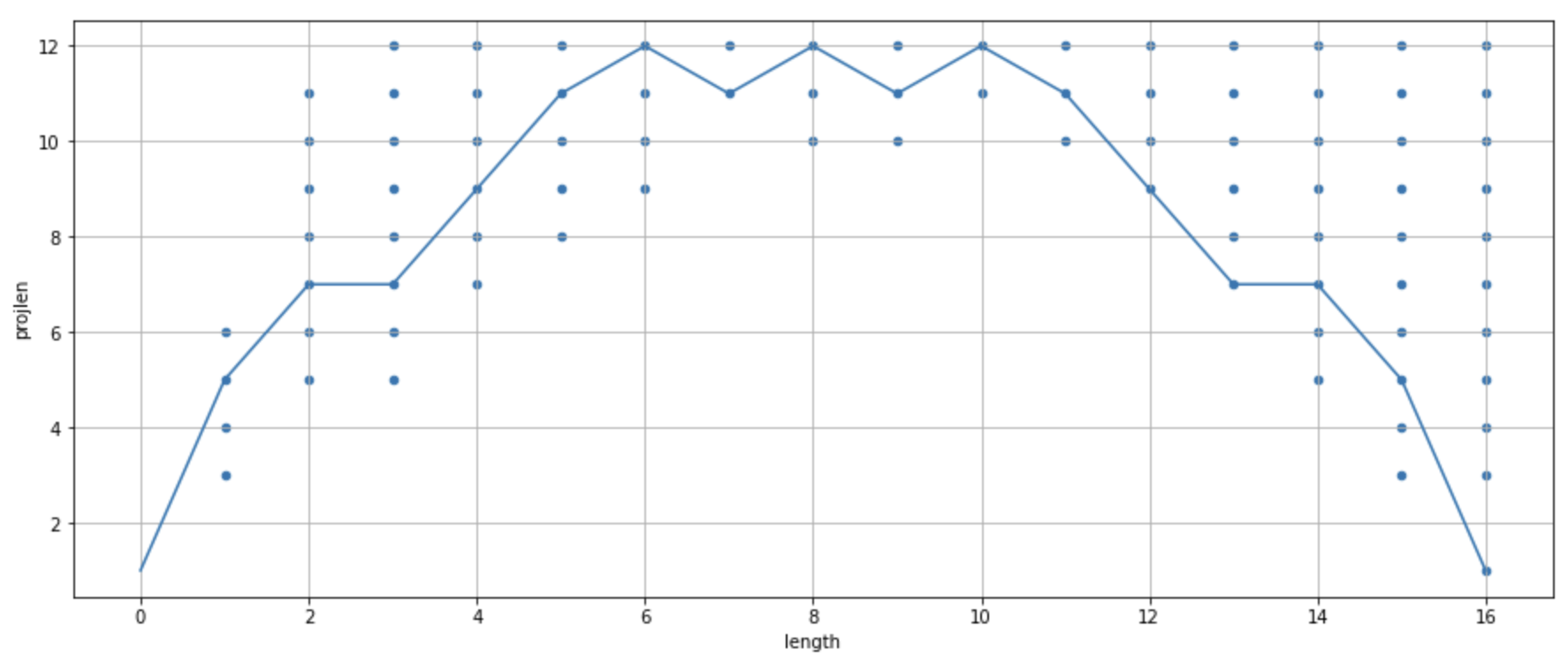}
\end{align*}
Now for each reservoir containing $\beta_i$ we can compute the
probability that a random sample would contain $\beta_i$. If we denote
this probability by $P(\beta_i)$, then
\[
P(\beta) = \prod_{i=1}^{16} P(\beta_i) = \prod_{i=1}^{16}
\frac{k}{\max( r(\beta_i), k)}
\]
where $k$ denotes the bucket size as above, and $r(\beta_i)$ denotes
the total number of elements seen by the bucket containing $\beta_i$.

With $k=500$ this probability is smaller than $10^{-6}$. Increasing
the bucket size to $10^5$ decreased this probability to around
$0.0005$. At this point it is believable that algorithmic improvements
(or simply a bigger computer and bigger buckets!) new kernel elements for $n=6$ might be
found. A similar analysis for $n=5$ shows that this case is considerably
harder.

\subsection{Future Directions}\label{sec:future-directions} We briefly comment on three avenues
which we believe warrant further exploration.

\subsubsection{Monte Carlo methods} The task of searching through
Garside normal forms in order to find kernel elements is a textbook
example of tree search. This problem features prominently in computer
Chess and Go. In Go, major progress was made in the 1990s and 2000s by
employing Monte Carlo evaluation strategies \cite{Go}: in order to
decide the strength of a board position, one randomly finishes the
game 5000 times and sees how often one wins. We tried a similar method
here:
\begin{enumerate}
\item A database $D$ of promising nodes (with score) is initiated with the
  identity element.
\item At each step, $\sigma$ is sampled from $D$ (randomly with
  weighting based on the score). We then do reservoir sampling starting from each 
  Garside suffix $\sigma \widetilde{u}$ for $N$ steps. Each suffix
  is then added to $D$, with score based on the lowest $\pl$ seen.
\end{enumerate}
(Thus, we regard exploring the tree of all possible Garside normal
forms as a one-player game, where the goal is to find kernel
elements, and a ``move'' consists of replacing the current element by
a Garside successor.)

We implemented this algorithm and it consistently found kernel
elements for $p=5$. We did several long runs in other settings (e.g. $(n=4,
p=7)$, $(n=5,p=41)$ etc.) without finding kernel elements. In this
setting there are several design choices (e.g. how to implement score and
select based upon it, how to choose $N$, \dots) which have a big
impact on performance. This is worth investigating systematically and
could lead to a much better algorithm.

\subsubsection{Machine learning methods} We begun this project with
the aim of employing machine learning methods. We outline here one
unsuccessful attempt to use vanilla supervised learning to improve our
sampling. Many other experiments are possible, and we hope that our
software can provide a good basis for further experiments.

Consider all elements of Garside length $\ell$. Imagine an
oracle $O$ which takes as input the matrix of an element of Garside
length $\ell$ and tells us the minimum $\pl$ amongst 
all Garside successors of Garside length $\ell + k$ for some value of
$k$ (e.g. $k=5$). Then it seems intuitive that using the oracle $O$
to weight reservoir sampling (i.e. elements with low projlen in $k$
further steps are more likely to be kept in a bucket) would lead to
better results.

With this motivation in mind, we tried to train a vanilla neural network
which takes as input the matrix of $\sigma$, and attempts to predict
the minimal $\pl$ in $k$ steps time. Our models attained reasonable
training accuracy ($\sim 80\%$), but seemed to make no difference at
all to the long-term performance of the algorithm.\footnote{Some further details on our algorithm:
  We encode a matrix $M$ of Laurent polynomials as a
  list of matrices encoding the coefficients of $v^i$. Because we only
  care about our matrices up to scalar multiple, we can assume that
  the non-zero matrices range from 0 to the $\pl$: $M_0, M_1,
  \dots, M_{\pl(M)}$. Now, after multiplying $M$ by all Garside
  successors of length $k$, it is easy to see that the resulting $\pl$s
  only depends on the first and last $k'$ matrices in our list, i.e. on $M_0, M_1,
  \dots, M_{k'}, M_{\pl(M)-k'+1}, \dots M_{\pl(M)}$ for some small
  value of $k'$. We further
  simplified our dataset by remembering only $M_0', M_1',
  \dots, M_{k'}', M_{\pl(M)-k'+1}', \dots M_{\pl(M)}'$ where $M_i'$ is
  a zero-one matrix whose entries
  record which entries in $M_i'$ are non-zero. This list was then
  serialized and used as input to a vanilla neural network.} In other
words, adding the neural network made no discernable difference to the
smallest $\pl$ found for large Garside lengths.
Thus, our attempt to use neural networks to spot patterns in Burau
matrices was unsuccessful. This suggests that such patterns are
either not there, or are difficult to detect.\footnote{One can also
  imagine that a drop in $\pl$ is a rare event,
  which is too rare to be picked up by the neural network.}

\subsubsection{Commutators} It is striking that all kernel elements
found in \cite{CL97,CL98,Bigelow99} are commutators. It is tempting to
try to modify our search regime to use this fact as a prior. It is not
clear how to do so, but a reasonable proposal could lead to better
results. We have not tried to express the elements that we have found
as (close relatives of) commutators.

\printbibliography
\end{document}